\documentclass[11pt]{amsart}

\usepackage{epigamath}

\usepackage[english]{babel}
\usepackage{amsmath,amsthm,microtype,amssymb,tikz,tikz-cd}

\numberwithin{equation}{section}

\usepackage{enumitem}

\theoremstyle{theorem}
\newtheorem{thm}{Theorem}[section]
\newtheorem{lem}[thm]{Lemma}
\newtheorem{cor}[thm]{Corollary}
\newtheorem{prop}[thm]{Proposition}

\theoremstyle{definition}
\newtheorem{dfn}[thm]{Definition}

\theoremstyle{remark}
\newtheorem{exa}[thm]{Example}
\newtheorem{rem}[thm]{Remark}

\newtheorem{quest}[thm]{Question}

\def\L{\mathcal L}
\def\M{\mathcal M}
\def\graded{}

\let\wtilde\widetilde
\let\bar\overline
\def\O{\mathcal O}
\let\b\beta
\def\F{\mathcal F}

\let\<\langle
\let\>\rangle

\def\inv{^{-1}}
\def\res#1{|_{#1}}

\def\m{\mathfrak m}
\def\Q{\mathbb Q}
\def\Z{\mathbb Z}
\def\C{\mathbb C}
\def\P{\mathbb P}
\def\w{\omega}

\DeclareMathOperator\End{End}
\DeclareMathOperator\Ext{Ext}
\DeclareMathOperator\Hom{Hom}
\DeclareMathOperator\Proj{Proj}
\DeclareMathOperator{\SL}{SL}
\DeclareMathOperator\Sym{Sym}
\DeclareMathOperator\rank{rank}
\DeclareMathOperator\Pic{Pic}
\DeclareMathOperator\Char{char}

\let\oxra\xrightarrow
\def\xrightarrow#1{\oxra{\ \,#1\ \,}}
\let\xra\xrightarrow
\let\lra\longrightarrow

\def\dlim{\varinjlim}

\makeatletter
\newcommand{\dashedrightarrow}[1][2pt]{%
  \settowidth{\@tempdima}{$\rightarrow$}\rightarrow
  \makebox[-\@tempdima]{\hskip-1.5ex\color{white}\rule[0.5ex]{#1}{1pt}}
  \phantom{\rightarrow}
}
\makeatother

\let\dra\dashrightarrow 

\def\set#1{\relpenalty=10000 \binoppenalty=10000\{#1\}}

\newcommand{\supth}[1]{\ensuremath{#1^{\mathrm{th}}}}

\EpigaVolumeYear{7}{2023} \EpigaArticleNr{21} \ReceivedOn{January 30, 2023}
\InFinalFormOn{July 1, 2023}
\AcceptedOn{August 2, 2023}

\title{Finite $\boldsymbol{F}$-representation type for homogeneous coordinate rings of non-Fano varieties}
\titlemark{FFRT for homogeneous coordinate rings of non-Fano varieties}

\author{Devlin Mallory}
\address{University of Utah, Department of Mathematics, 
155 South 1400 East, Salt Lake City, UT 84112, USA}
\email{malloryd@math.utah.edu}

\authormark{D.~Mallory}

\AbstractInEnglish{Finite $F$-representation type is an important notion in characteristic $p$ commutative algebra, but explicit examples of varieties with or without this property are few.
We prove that a large class of homogeneous coordinate rings in positive characteristic will fail to have  finite $F$-representation type. 
To do so, we prove a connection between differential operators on the homogeneous coordinate ring of $X$ and the existence of global sections of a twist of $(\Sym^m \Omega_X)^\vee$. By results of Takagi and Takahashi, this allows us to rule out finite $F$-representation type for coordinate rings of varieties with $(\Sym^m \Omega_X)^\vee$ not ``positive.''  By using positivity and semistability conditions for the (co)tangent sheaves, we show that several classes of varieties fail to have  finite $F$-representation type, including many Calabi--Yau varieties and complete intersections of general type. Our work also provides examples of the structure of the ring of differential operators for non-$F$-pure varieties, which to this point have largely been unexplored.} 


\acknowledgement{This material is based upon work supported by the National Science Foundation under Grant No.~\#1840190.}

\begin{document}


\maketitle

\begin{prelims}

\DisplayAbstractInEnglish

\bigskip

\DisplayKeyWords

\medskip

\DisplayMSCclass

\end{prelims}


\newpage

\setcounter{tocdepth}{1}

\tableofcontents


\section{Introduction}
For a ring $R$ of positive characteristic $p$, the module-theoretic properties of the Frobenius pushforwards $F_e^* R$ capture a great deal of information about $R$. 
For example, when $R$ is local and $F$-finite, the following hold: 
\begin{itemize}
\item $F_*^eR $ is locally free if and only if $R$ is regular.
\item $F_*^eR $ has a free $R$-summand if and only if $R$ is $F$-split.
\item The limit as $e\to \infty$ of the number of free summands of $F_*^eR$, divided by the rank $p^{e\dim R}$, is the $F$-signature of $R$, a subtle invariant of the singularities of $R$.
\end{itemize}
Each of these properties is a statement about the decomposition of $F_*^eR$ into indecomposable summands.
It is then natural to ask which summands of the $F_*^e R$ occur as $e$ varies over all of $\mathbb N$.
If only finitely many summands occur, we say that $R$ has \emph{finite $F$-representation type}, often abbreviated as FFRT.

The notion of FFRT was introduced in \cite{SVdB}, in the study of the simplicity of rings of differential operators, and has since found a range of applications in commutative algebra.

\begin{exa}
Regular local rings have FFRT since for a regular local ring $R$, we have $F_*^e R\cong R^{\oplus p^{e\dim R}}$.
Quadric hypersurface rings have FFRT since if $R$ is such a ring, then $F_*^e R$ is a Cohen--Macaulay module, and $R$ has only finitely many isomorphism classes of indecomposable Cohen--Macaulay modules (see \cite{BE}).
\end{exa}

FFRT is a strong condition on a ring $R$. For example, a ring $R$ with FFRT will have only finitely many associated primes of the local cohomology modules $H^i_I(R)$, for any ideal $I\subset R$
(see \cite{TT,HNB,DQ}).
If $R$ is the homogeneous coordinate ring of a smooth curve of genus $g$ in some embedding in projective space, $R$ will have FFRT if and only $g=0$.
(In contrast, \cite{Shibuta} showed that the  1-dimensional local rings will have FFRT over an algebraically closed or finite field.)

The study of the decompositions of $F^e_* R$ is also closely related to questions arising in algebraic geometry in positive characteristic. When $R$ is the homogeneous coordinate ring of a projective variety $X$ embedded by a very ample line bundle $\L$, $F^e_* R$ decomposes into the direct sum of the pushforwards $F^e_* \L^{i}$ for $0\leq i <p^e$. The FFRT property of a variety $X$, and even just the study of the decomposition of $F_*^e\O_X$, is of great interest. For example, the following hold: 
\begin{itemize}
\item When $X$ is an abelian variety, the decomposition of $F_*^e\O_X$ reflects the $p$-rank of $X$ (\textit{e.g.}, whether $X$ is ordinary or supersingular); see \cite{ST}.
\item If $X$ is a projective variety, $F_*L$ decomposes as the direct sum of line bundles for any invertible sheaf $L$ if and only if $X$ is a smooth toric variety; see \cite{Thomsen,Achinger}.
\item When $X$ is a del Pezzo surface of degree $5$, the summands of $F^e_*\O_X$ provide interesting new examples of indecomposable vector bundles; see \cite{Hara1}.
\item The fact that a Grassmannian of 2-dimensional quotients has FFRT, established in \cite{RSB}, reflects subtle representation-theoretic information about $\SL_2$ over a field of positive characteristic.
\end{itemize}
Whether a projective variety has a coordinate ring with FFRT is then interesting  from the point of view of not just commutative algebra, but also algebraic geometry and representation theory.
\looseness-1

Although the expectation is that FFRT should be somewhat rare, evidence and examples in higher dimension are rare, even when just considering homogeneous coordinate rings of smooth projective varieties. For example, for Fano surfaces, the following is known. 

\begin{exa}
Let $X_d$ be a del Pezzo surface of degree $d$. If $d\geq 6$, $X_d$ is toric, and thus its homogeneous coordinate rings are direct summands of polynomial rings, and hence have FFRT by 
\cite[Proposition~3.1.6]{SVdB}.
In \cite{DM2}, we show that the homogeneous coordinate rings of degree 5 del Pezzo surfaces also have FFRT.
If $d\leq 4$, then FFRT is not known for the homogeneous coordinate rings of $X_d$; in particular, it is unknown whether the hypersurface ring $k[x,y,z,w]\,/(x^3+y^3+z^3+w^3)$ has FFRT.
\end{exa}

This example concerns Fano varieties; for such varieties, the homogeneous coordinate rings will have ``mild'' (\textit{i.e.}, strongly $F$-regular) singularities. 
However, it is known that simply having mild singularities does not guarantee FFRT:
As pointed out in
\cite[Remark~3.4]{TT}, the example of \cite{SinghSwanson}, combined with the results of \cite{TT}, yields a strongly $F$-regular hypersurface UFD which does not have FFRT.

If one leaves the strongly $F$-regular setting, one might expect the FFRT property to be even rarer. The following are among the few previously known examples. 

\begin{exa}
In \cite[Theorem~1.2]{ST}, Sannai and Tanaka show that if $A$ is an abelian variety of $p$-rank $\mu_A$, then the indecomposable summands of $F^e_*\O_A$ have rank $p^{e(d-\mu_A)}$.
If $\mu_A<d$, then there are infinitely many indecomposable summands of $F_*^e\O_A$ as $e$ varies. They show, moreover, that when $\mu_A=d$, there are infinitely many distinct line bundles appearing in $F_*^e\O_A$. In either case, then,
this implies immediately that $A$ does not even have FFRT for $\O_A$, which implies that the homogeneous coordinate ring does not have \graded FFRT.
\end{exa}

\begin{exa}
In \cite[Theorem~7.2]{HO}, Hara and Ohkawa show that 2-dimensional graded surfaces obtained as section rings of $\Q$-divisors on $\P^1$ will fail to have FFRT if they are not log-terminal (and thus not $F$-regular), except potentially when the characteristic $p$ divides the coefficients of the $\Q$-divisor. 
\end{exa}

In this paper, we further confirm this expectation, by ruling out FFRT for the coordinate rings of several classes of varieties. 

\begin{thm}
Let $X$ be one of the following varieties over a perfect field $k$ of positive characteristic:
\begin{itemize}
\item a non-uniruled Calabi--Yau variety with $H^i(X,\O_X)=0$ for $i=1,\dots,\dim X-1$, 
\item a non-unirational K3 surface, 
\item a general-type complete intersection in $\P^N$ of dimension $\geq 3$.
\end{itemize}
Let $\L$ be a very ample line bundle on $X$ such that $\L^{\otimes r}\cong \w_X$ for some $r\in \Z$. Then the homogeneous coordinate ring 
$$
R(X,\L):=\bigoplus H^0(X,\L^{\otimes m})
$$
does not have \graded FFRT.
\end{thm}

We also give a new proof of the fact that the homogeneous coordinate ring of a smooth curve of genus $g$ has FFRT if and only if $g=0$.

As just one concrete example of this theorem, we have the following. 

\begin{exa}
As we discuss in Section~\ref{K3}, if $p\equiv 1\mod 4$, the above says that 
$$
\frac{k[x,y,z,w]}{x^4+y^4+z^4+w^4}
$$
does not have \graded FFRT. (The $p\equiv 3 \mod 4$ case is still unknown to us.)

Similarly, for \emph{any} characteristic $p>0$  and any $d\geq 5$,
we show in Section~\ref{generaltype}
that the ring
$$
\frac{k[x,y,z,w,t]}{x^d+y^d+z^d+w^d+t^d}
$$
does not have \graded FFRT.
\end{exa}

We finish the introduction with a brief sketch of the proof, as well as an outline of the paper. 
In Section~\ref{TTdif}, we use results of \cite{TT} to connect the \graded FFRT property of $R$ to the behavior of certain modules under  the ring of differential operators on $R$. 
To study the ring of differential operators on a homogeneous coordinate ring $R(X,\L)$ of a variety $X$, 
in Section~\ref{diffcone}, we prove the following, which is the main technical result of this paper. 

\begin{thm}\label{thm17}
Let $k$ be a perfect field, and
let $R$ be a Gorenstein $\mathbb N$-graded $k$-algebra domain of finite type over $R_0=k$.
Let $X=\Proj R$, with very ample line bundle $\L=\O_X(1)$.
If\, $R$ has differential operators of negative degree, then $H^0((\Sym^m  \Omega_X)^\vee \otimes\L\inv)\neq0$ for $m\gg0$.
\end{thm}

This result, which may be of independent interest, provides a new tool for studying differential operators in positive characteristic; for example, this allows one to show that homogeneous coordinate rings of curves of genus $g\geq 1$ are not $D$-simple.

Combining this with \cite{TT}, we obtain the following. 

\begin{thm}
Retain the assumptions of Theorem~\ref{thm17}.
If\, $H^0((\Sym^m \Omega_X)^\vee\otimes \L\inv)=0$ for all $m$, then $R$ does not have \graded FFRT.
\end{thm}

Using this theorem, we then reduce the study of \graded FFRT for section rings $R(X,\L)$ to the ``nonpositivity'' properties of $(\Sym^m \Omega_X)^\vee$. For Calabi--Yau varieties, we use the results of \cite{Langer} to understand their cotangent bundles in Section~\ref{K3}. In Section~\ref{generaltype}, we use \cite{Noma1,Noma2} to understand what happens for complete intersections of general type.
We finish by describing some open questions in Section~\ref{questions}.

\subsection{Acknowledgments}

We would like to thank Eamon Quinlan-Gallego for introducing us to the
results of \cite{TT}, which form the basis of our methods in this
paper, and for helpful suggestions on a draft of this paper.  We would
also like to thank Swaraj Pande and Shravan Patankar for useful
discussions.  Finally, we would like to thank Karl Schwede, Anurag
Singh, Karen Smith, and Jack Jeffries for useful comments on an
earlier draft of the paper.

\section{Preliminaries}
Throughout, we will work over a perfect field $k$ of characteristic $p>0$.
We will try to specify throughout, but the graded rings we will consider will be finitely generated $\mathbb N$-graded $k$-algebras $R$ with $R_0=k$; we do not require $R$ to be generated in degree 1.

\subsection{Finite $\boldsymbol{F}$-representation type}

\begin{dfn}
  Let $R$ be a ring of characteristic $p>0$, and assume that $R$ is $F$-finite (\textit{i.e.}, the Frobenius morphism is finite). For $e>0$, we denote by $F^e$ the $\supth{e}$ iterated Frobenius and write $F^e_* R$ for the $R$-module obtained by restriction of scalars along $F^e\colon R\to R$.
We say that $R$ has finite $F$-representation type (or FFRT) if there are finitely many finitely generated $R$-modules $M_1,\dots,M_N$ such that for any $e$, we can write
$$
F_*^e R \cong M_1^{\oplus a_{e,1}}\oplus \dots\oplus
M_N^{\oplus a_{e,N}}
$$
(as $R$-modules)
for some nonnegative $a_{e,i}$. 
That is, $R$ has FFRT if there are only finitely many modules occurring in irreducible $R$-module decompositions of each $F_*^e R$.

We say that a graded ring $R$ has \emph{graded} FFRT if there are
finitely many finitely generated $\Q$-graded $R$-modules $M_1,\dots,M_N$ such that for any $e$, we can write
$$
F_*^e R \cong 
\bigoplus_{i=1}^N
\bigoplus_{j=1,\dots,n_{e,i}}
M_i^{a_{e,i,j}}\left(\b_{e,i,j}\right)
$$
for some  $\b_{e,i,j}\in \Q$ and nonnegative $a_{e,i,j}$, where $M(\b)$ denotes the shift of the $\Q$-graded module $M$ by $\b$.
That is, only finitely many graded modules occur, \emph{up to shifting the grading}, in 
an irreducible decomposition of the $F_*^e R$.
\end{dfn}

Some authors use simply ``FFRT'' for a graded ring to refer to graded FFRT. The following lemma, likely known to experts, ensures that no confusion results, and we will henceforth simply say ``FFRT.''
\looseness-1

\begin{lem}
Let $k$ be a perfect field.
If an $\mathbb N$-graded domain $R$ finitely generated over $R_0=k$ has FFRT, it has graded FFRT.
\end{lem}

\begin{proof}
Let $\m$ be the homogeneous maximal ideal.
If $R$ has FFRT, so does the localization $R_\m$.
Let 
$$
F_*^e R = M_{e,1}^{\oplus a_{e,1}}\oplus \dots \oplus M_{e,n_e}^{a_{e,n_e}}
$$
be the decomposition into irreducibles. 
Since $R_\m$ has FFRT, only finitely many indecomposable modules $N_1,\dots,N_r$ occur up to isomorphism in the decompositions of $F_*^e R_\m$.
Thus, we must have that each $M_{ij}$ is isomorphic after localizing at $\m$ to one of the finitely many $N_r$. If there were infinitely many isomorphism classes (up to shifts) of the $M_{ij}$, then non-isomorphic $M_{ij}$, $M_{i'j'}$ would become isomorphic after localizing.
Noting that the $M_{ij}$ are finitely generated and $R$-torsion-free, the following two lemmas then conclude the proof.

\begin{lem}
Let $R$ be an $\mathbb N$-graded ring with homogeneous maximal ideal $\m$ $($\textit{i.e.}, $R_0$ is a field\,$)$.
Let $M,N$ be graded $R$-modules, with $M$ and $ N$ finitely generated and $N$ torsion-free. If there is a morphism $\phi\colon M_\m\to N_\m$ of\, $R_\m$-modules, there is a morphism $\wtilde \phi\colon M\to N$ of\, $R$-modules with $\wtilde \phi \otimes _R R_\m = \phi$, up to a unit in $R_\m$.
\end{lem}

\begin{proof}
Say $M$ is generated by $m_i$ for $i=1,\dots,e$, and $N$ by $n_i$, and that $\phi(m_i)=\sum (a_{ij}/b_{ij}) n_j$ with  $a_{ij}\in R$, $b_{ij} \in R- \m$. Multiplication by the unit $\prod b_{ij}$ of $R_\m$ is an isomorphism, so we may as well assume that we can write $\phi(m_i)=\sum a_{ij} n_j$ with $a_{ij}\in R$.

Let
$$
R^f\xra{\psi} R^e \lra M \lra 0
$$
be a homogeneous presentation of the graded $R$-module $M$, with
$e_i$ the basis element of $R^e$ mapping onto $m_i$. 
The assignment $e_i\mapsto \sum a_{ij} n_j$ defines a map $ \Phi\colon R^e\to N$. 
Since $R^e$ is a finite $R$-module and $N$ is finitely generated and homogeneous, $\Hom(R^e,N)$ is the sum of its graded components, so we may write $ \Phi =  \Phi_d +\dots + \Phi_{d+\ell}$ for some $d$, with each $ \Phi_i$ homogeneous of degree $i$.

We claim that each $ \Phi_i\colon R^e\to N$ descends to a map $M\to N$.
This occurs if and only if each $ \Phi_i \circ \psi$ is the zero map. 
Since $\phi\colon M_\m\to N_\m$ is defined, we know that $\Phi \circ \psi $ is the zero morphism $R_\m^f\to N_\m$. Thus, for any $r$ homogeneous in $R^f$, we have
$$
0=\Phi\circ \psi(r) =\Phi_d(\psi(r))+\dots+\Phi_{d+\ell}(\psi(r))
$$
(using that the left side is zero in $N_\m$, and thus in $N$, by the torsion-freeness of $N$).
Each term on the right side occurs in a different degree, and thus each graded piece on the right side is 0. Thus, all $\Psi_i\circ \psi$ are~0, and all $\Phi_i$ descend to morphisms $\phi_i\colon M\to N$. Finally, taking $\wtilde \phi\colon\phi_d+\dots+\phi_{d+\ell}$, we obtain the desired map.
\end{proof}

\begin{lem}
Let $R$ be an $\mathbb N$-graded ring with homogeneous maximal ideal $\m$.
Let $\phi\colon M\to N$ be an $R$-linear morphism $($not necessarily graded\,$)$ of finitely generated torsion-free graded $R$-modules, which is an isomorphism $M_\m\to N_\m$. If $\phi=\phi_d+\dots+\phi_{d+\ell}$ is the decomposition of $\phi$ into graded pieces, then $\phi_d\colon M\to N$ is an isomorphism.
\end{lem}

\begin{proof}
The surjectivity is clear: since it is graded, $\phi_d\colon M\to N$ is surjective if and only if  $M/\m M\to N/\m N$ is surjective, but $\phi$ and $\phi_d$ induce the same map $M/\m M\to N/\m N$ (since the $\phi_{d+i}$ all have image inside $\m N$).
For the injectivity, we can apply the same argument as we did for the map $\phi\colon M\to N$ to the inverse $\psi:=\phi\inv\colon N\to M$. We then have that the first graded piece $\psi_{d'}$ gives a graded surjection $N\to M$. We then have that $\psi_{d'}\circ \phi_d\colon M\to M$ is surjective; Nakaya's lemma then implies that $\psi_{d'}\circ \phi_d$ is also injective, and thus $\phi_d$ is injective.
\end{proof}

With the lemma established, we see that any $M_{i,j}$ becoming isomorphic after localizing at $\m$ already were, and so only finitely many isomorphism classes of the $M_{i,j}$ occur.
\end{proof}

\subsection{Differential operators}
We briefly recall the notion of differential operators. 

\begin{dfn}
Let $k$ be a field, and 
let $R$ be a $k$-algebra. We define $D^m_{R/k}\subset \End_k(R)$, the $k$-linear differential operators of order $m$, inductively as follows:
\begin{itemize}
\item $D^0_{R/k} =\Hom_R(R,R)\cong R$, thought of as multiplication by $R$. 
\item  $\delta \in \End_k(R)$ is in $D^m_{R/k}$ if $[\delta, r] \in D^{m-1}_{R/k}$ for any $r\in D^0_{R/k}$.
\end{itemize}
We write $D_{R/k}=\bigcup D^m_{R/k}$. Then $D_{R/k}$ is a noncommutative ring, and $R$ is a left $D_{R/k}$-module.
\end{dfn}

In characteristic $p$, the behavior of differential operators can be quite different from in characteristic $0$. For example, if $R=\mathbb F_p[t]$, $D_{R/k}$ is not noetherian or finitely generated over $R$.
In spite of this, in some ways the ring of differential operators is \emph{better} understood in positive characteristic, through its links to the Frobenius morphism. In particular, there is the following description. 

\begin{prop}[\textit{cf.}~\protect{\cite[Theorem~1.4.9]{Yekutieli}}]
\label{Yek}
Let $k$ be a perfect field of characteristic $p>0$, and let $R$ be an $F$-finite $k$-algebra. Then 
$$
D_{R/k} = \bigcup_{e} \Hom_{R^{p^e}}(R,R).
$$
\end{prop}

Note that when $k\subset k'$ are both perfect fields and $R$ is both a $k$- and $k'$-algebra, this implies that $D_{R/k}=D_{R/k'}$. When $k$ is clear, we will often write just $D_R$ for $D_{R/k}$.

\subsection{Extended cotangent bundles}
\label{extcot}
Let $X\subset \P_k^N$ be a subvariety, with homogeneous coordinate ring $R$.
There are two sheaves on $X$ (or equivalently, graded modules on $R$) associated to $X\subset \P^N$: there is the usual cotangent sheaf $\Omega_X$, which is intrinsic to $X$, and there is the \emph{extended} cotangent sheaf $\wtilde \Omega_X$, an extension of $\Omega_X$ by $\O_X$, which depends on the embedding. In this section, we recall the definition and relation between these two sheaves, following for the most part the exposition in \cite[Section~1.2]{BDO}.

\begin{dfn}
Let $\Omega_{R/k}$ be the module of K\"ahler differentials of the graded $k$-algebra $R$. If $0\to I\to S\to S/I=R\to 0$ is a presentation of $R$ as a quotient of a polynomial ring $S$ (\textit{e.g.}, the one corresponding to $X\subset \P^N$), then $\Omega_R$ is the cokernel of $I/I^2 \to \Omega_S\otimes_S R$, $f\mapsto df$.
The corresponding sheaf on $X=\Proj R$ is denoted by $\wtilde \Omega_X$ and called the extended cotangent sheaf of $X$.
\end{dfn}

\begin{rem}
It is immediate to see that $\wtilde \Omega_X$ is not the same sheaf as $\Omega_X$: if $\dim X=n$, then $\Omega_X$ has generic rank $n$, while $\wtilde \Omega_X$ has generic rank $n+1$ (as it corresponds to the module of K\"ahler differentials of a ring of dimension $n+1$).
\end{rem}

Now, we recall the connection between $\Omega_X$ and $\wtilde \Omega_X$ in more detail. 
We have presentations of $\Omega_X$ and $\wtilde \Omega_X$, respectively:
$$
I_X/I_X^2 \lra \Omega_{\P^N}\res X \lra \Omega_X\lra 0
$$
and
$$
I_X/I_X^2 \lra \O_X^{N+1}(-1) \lra \wtilde\Omega_X\lra 0.
$$
There is also a natural short exact sequence $0\to \Omega_{\P^n}\res X \to \O_X^{N+1}(-1)\to \O_X\to 0$ coming from the restriction of the Euler sequence on $\P^n$. These exact sequences induce morphisms (depicted with dashed lines) fitting into the diagram
$$
\begin{tikzcd}
 & I_X/I_X^2 \ar[d]\ar[r,equal]& I_X/I_X^2\ar[d]  &  & \\
 0\ar[r] & \Omega_{\P^N}\res X\ar[d]\ar[r]&\O_X^{N+1}(-1) \ar[d]\ar[r] & \O_X\ar[r]\ar[d,equal] & 0 \\
 0\ar[r] & \Omega_X \ar[d]\ar[r,dashed]&\wtilde \Omega_X \ar[r,dashed] \ar[d]& \O_X\ar[r] & 0 \\
 & 0 & 0 &  &  
\end{tikzcd}
$$
with exact rows and columns.
In particular, we highlight the short exact sequence
\begin{equation}
0\lra \Omega_X\lra \wtilde \Omega_X\lra \O_X\lra0.
\end{equation}

\begin{rem}
One can ask precisely how $\wtilde \Omega_X$ depends on the choice of embedding $X\subset \P^N$. 
From the discussion above, we know that $\wtilde \Omega_X$ is an extension of $\Omega_X$ by $\O_X$, each of which are independent of the embedding. The set of nontrivial extensions is parametrized by lines in $\Ext^1(\O_X,\Omega_X)=H^1(X,\Omega_X)$.
So we have, for example, that for curves, the sheaf  $\wtilde \Omega_X$ is in fact independent of the embedding $X\subset \P^N$ since $H^1(X,\Omega_X)=\C$ for any smooth projective curve $X$.
\vadjust{\goodbreak}
\end{rem}

\subsection{Global sections of duals of trivial extensions}
\def\N{\mathcal N}
The following lemma will be of use to us in Section~\ref{diffcone} in relating $\wtilde \Omega_X$ and $\Omega_X$.  

\begin{lem}
\label{extens}
Let $X$ be any variety, $\L$ an ample line bundle, and $ \M$ a locally free sheaf given as an extension $0\to \N \to \M\to \O_X\to0$ with $H^0((\Sym^m \N)^\vee \otimes \L\inv)=0$ for all $m$.
Then $H^0((\Sym^m  \M)^\vee\otimes \L\inv)=0$ for all $m$.
\end{lem}

\begin{proof}
First, recall (\textit{e.g.}, from \cite[Exercise~II.5.16]{Hartshorne}) that $\Sym^m  \M$ has a filtration by locally free sheaves
$$
0=\M_{m+1}\subset \dots \subset \M_1\subset \M_0 = \Sym^m \M
$$
with quotients
$$
\M_i/\M_{i+1}\cong \Sym^i(\N)\otimes \Sym^{m-i}(\O_X)\cong \Sym^i \N.
$$
Consider 
the first step
$$
0\lra \M_1\lra \Sym^m \M \lra \O_X\lra 0.
$$
Dualizing and twisting by $\L\inv$, we get
$$
0\lra 
\L\inv \lra
(\Sym^m \M)^\vee\otimes \L\inv \lra \M_1^\vee \otimes \L\inv \lra0, 
$$
with associated long exact sequence
$$
0\lra 
H^0\left(\L\inv\right) \lra
H^0\left(\left(\Sym^m \M\right)^\vee\otimes \L\inv\right) \lra H^0\left(\M_1^\vee \otimes \L\inv \right)\lra \cdots . 
$$

Clearly, 
$H^0(\L\inv)=0$, and so we get the desired vanishing if 
$H^0(\M_1^\vee \otimes \L\inv)=0$ as well.

The next short exact sequence coming from the filtration is
$$
0\lra \M_2\lra \M_1\lra \Sym^1\N =\N\lra 0.
$$
Again, dualizing and twisting by $\L\inv$, we reduce the vanishing of 
$H^0(\M_1^\vee \otimes \L\inv)$ to that of 
$H^0(\M_2^\vee \otimes \L\inv)$.

Continuing  in this fashion, we reduce the vanishing to that of 
$H^0(\M_m^\vee \otimes \L\inv)$; however, $\M_m^\vee$ is itself $\Sym^m \N$, so the vanishing is clear, and thus the lemma follows.
\end{proof}

\begin{rem}
The same is true if we replace $\O_X$ with $\O_X^{\oplus N}$ for any $N$, \textit{i.e.}, whenever $\M$ is an extension of $\N$ by the trivial sheaf. We need only the case $N=1$, which simplifies the notation slightly in the above.
\end{rem}

\section{Differential operators on \texorpdfstring{$\boldsymbol{R[1/x]}$}{$R[1/x]$}}
\label{TTdif}

Let $k$ be a perfect field, and
let $R$ be an $\mathbb N$-graded domain of finite type over $R_0=k$.
In this section, we consider a necessary condition for $R[1/x]$ to be a simple $D_R$-module, and in fact for $R[1/x]$ to be a finitely generated $D_R$-module; this then provides a necessary condition for $R$ to have \graded FFRT by results of \cite[Corollary~2.10]{TT}.
Recall that if $R$ is graded, then so is $D_R$, where an element $\delta \in D_R$ has degree $e$ if $\delta(R_m)\subset R_{m+e}$ for all $m$.

We begin by recalling the following theorem of \cite{TT}, which connects the \graded FFRT property for graded rings to  properties of the $D_R$-module $R[1/x]$.

\begin{thm}[\textit{cf.}~\protect{\cite[Corollary~2.10]{TT}}]
\label{TTThm}
Let $k$ be a perfect field, and
let $R$ be an $\mathbb N$-graded domain of finite type over $R_0=k$.
If\, $R$ has \graded finite $F$-representation type, then for any nonzerodivisor $x\in R$, $R[1/x]$ is generated by $1/x$ as a $D_{R/k}$-module.
\end{thm}

The following statement follows  essentially from the fact that the $D_R$-module structure on $R[1/x]$ is compatible with the grading on $D_R$; we assume this is standard for experts but include a proof for completeness.

\begin{prop}
\label{fgneg}
Let $x\in R$ be an element of positive degree.
If\, $R[1/x]$ is finitely generated over $D_R$, then $D_R$ has elements of negative degree.
\end{prop}

\begin{proof}
The ring  $R[1/x]$ has a strictly increasing filtration by $R$-modules:
  $$
  R\subset R\<1/x\> \subset R\<1/x^2\>\subset \cdots .
  $$
If $D_R \cdot R\<1/x^n\> \subset R\<1/x^n\>$ for all $n$, then clearly this is also a strictly increasing filtration by $D_R$-modules. In this case, $R[1/x]$ is clearly not finitely generated over $D_R$ since a generating set would need to include elements with arbitrarily negative degree.
So, we must have that $D_R\cdot (r/x^n) = 1/x^{n'}$ for some $r,n,n'$ with $x^n\nmid r$ and $n'>n$.
We claim that this implies that $D_R$ has elements of negative degree.

To see this, we recall the $D_R$-module structure on $R[1/x]$ from \cite[Example~2.6]{TT}: First, recall that $D_R = \bigcup_{e} \Hom_{R^{p^e}} (R,R)$ (with identical grading).
For $\delta \in \Hom_{R^{p^e}}(R,R)\subset D_R$ and $r/x \in R[1/x]$, we have that
$$
\delta\left(\frac r{x^n}\right)
=
\delta\left(\frac{ rx^{p^e-n}}{x^{p^e}}\right) = 
\frac{\delta\left(rx^{p^e-n}\right)}{x^{p^e}}.
$$
The right side clearly has degree 
$$
\deg\left(rx^{p^e-n}\right) + \deg\delta-\deg\left(x^{p^e}\right)=
\deg(r)+\deg\left(x^{p^e-n}\right) + \deg\delta-\deg\left(x^{p^e}\right)
=\deg(r)-n\deg(x)+\deg(\delta),
$$
while $\deg(r/x^n)=\deg r -n\deg x$.
For $\delta(r/x^n)$ to be equal to $1/x^{n'}$, which has degree
$$
-n'\deg x< -n\deg x,
$$
we must have that $\deg \delta <0$.
\end{proof}

\begin{rem}
Proposition~\ref{fgneg} is of course true in characteristic 0, with a slightly different proof: 
One can write $R=S/I$ for $S$ a polynomial ring, and thus view differential operators on $R$ as the quotient of the differential operators on $S$ preserving $I$. Since $S$ is regular, the latter are compositions of derivations, and it is easy to check the relation between the degree of differential operators and their action on $R[1/x]$ directly for derivations by the quotient rule.
\end{rem}

Combining Theorem~\ref{TTThm} with Proposition~\ref{fgneg}, we immediately have the following. 

\begin{cor}
\label{fgnegcor}
Let $k$ be a perfect field
and $R$ an $\mathbb N$-graded $k$-algebra of finite type over $R_0=k$; assume moreover that $R\neq R_0$.
If\, $D_R$ has no elements of negative degree, then $R$ does not have \graded FFRT.
\end{cor}

When $R$ is strongly $F$-regular, this corollary was known and follows from \cite[Theorem~4.2.1]{SVdB}, which says that if $R$ is strongly $F$-regular and has FFRT, then $D_R$ is a simple ring. Then $R$ must be a simple $D_R$-module and thus have differential operators of negative degree.
Thus, it is the non-strongly-$F$-regular case that is new, and that is the case we will use in this paper.

\section{Differential operators on cones over varieties}
\label{diffcone}

The following result is known for smooth projective varieties over the complex numbers. 

\begin{thm}[\textit{cf.}~\protect{\cite{Hsiao}}]
Let $R$ be a graded $\C$-algebra of finite type over $R_0=\C$, with isolated singularity at the homogeneous maximal ideal. Let $X=\Proj R$. If\, $R$ has differential operators of negative degree and $\L$ is any ample line bundle, then $H^0((\Sym^m T_X)\otimes\L\inv)\neq0$ for $m\gg0$.
\end{thm}

We note that the property that 
$H^0((\Sym^m T_X)\otimes\L\inv)\neq0$ is equivalent to \emph{bigness} of the vector bundle $T_X$ (and is thus independent of the particular choice of ample line bundle $\L$).

In this section, we prove a related result, which lets us remove the characteristic 0 assumption and  allows for  singularities on $X$, but at the cost of an arithmetic Gorenstein condition. 

\begin{thm}
\label{diffops}
Let $k$ be a perfect field, and
let $R$ be a Gorenstein graded  domain of finite type over $R_0=k$.
Let $X=\Proj R$, with very ample line bundle $\L=\O_X(1)$.
If\, $R$ has differential operators of negative degree, then $H^0((\Sym^m  \Omega_X)^\vee \otimes\L\inv)\neq0$ for $m\gg0$.
\end{thm}

\begin{proof}
First, we show that the vanishing of $H^0((\Sym^m\wtilde \Omega_X)^\vee\otimes \L\inv)$ implies that $R$ has no differential operators of negative degree. Recall from Section~\ref{extcot} that $\wtilde \Omega_X$ is just the sheaf corresponding to the $R$-module $\Omega_{R/k}$.
We begin by recalling the equality of graded $R$-modules
$$
D_{R/k} = H^d_\Delta(R\otimes_k R)(a),
$$
where $d=\dim R$, $a$ is the $a$-invariant of $R$, and $\Delta$ is the kernel of the multiplication map $R\otimes R\to R$. This is \cite[Theorem~A]{Jeffries} (note that the formula there  does not make explicit the shift in grading); this is where the requirement that $R$ be Gorenstein is needed.

We can rewrite the right-hand side as 
$$
H^n_\Delta(R\otimes_k R)(a) = \dlim_\ell \Ext^d\left(\frac{R\otimes_k R}{\Delta^\ell},R\otimes_k R\right)(a).
$$
By graded local duality for the Gorenstein ring $R\otimes_k R$, we can write 
$$
\Ext^d\left(\frac{R\otimes_k R}{\Delta^\ell},R\otimes_k R\right) (a)
= H^d_\m \left(\frac{R\otimes_k R}{\Delta^\ell}\right)^*(-a)
$$
for each $\ell$ (note that since $a$ is the $a$-invariant of $R$, $2a$ is the $a$-invariant of $R\otimes_k R$), and $\m$ is the homogeneous maximal ideal of $R\otimes_k R$.
To show that $D_{R/k}$ has no differential operators of negative degree, it then suffices to show that
$H^d_\m\left(\frac{R\otimes_k R}{\Delta^\ell}\right)^*(-a)$ is zero in degree less than $0$ for all $\ell$.

We proceed on induction on $\ell$. For $\ell=1$, there is no need to dualize: using that $(R\otimes_k R)/\Delta= R$, we have
$$
\Ext^d((R\otimes_k R)/\Delta,R\otimes_k R)(a)
=
\Ext^d((R\otimes_k R)/\Delta,\w_{R\otimes_k R}) (-a) = \w_{R}(-a)=R,
$$
which clearly has no elements of negative degree.
(This perhaps obscures what is going on: this just reflects the fact that $D^0_{R/k}=R$.)

The modules $(R\otimes_k R)/\Delta^\ell$ fit into short exact sequences
$$
0\lra
\Delta^\ell/\Delta^{\ell+1}
\lra
(R\otimes_k R)/\Delta^{\ell+1} \lra 
(R\otimes_k R)/\Delta^\ell\lra0, 
$$
which on local cohomology give 
$$
H^{d-1}_\m\left((R\otimes_k R)/\Delta^\ell\right)
\lra
H^d_\m\left(\Delta^\ell/\Delta^{\ell+1}\right)
\lra
H^d_\m\left((R\otimes_k R)/\Delta^{\ell+1}\right) \lra 
H^d_\m\left((R\otimes_k R)/\Delta^\ell\right)\lra0.
$$
(All modules involved in the short exact sequence have support contained in $\Delta$ and thus have dimension at most $d$, so there are no terms in degree greater than $d$.)
Set 
$$
K:=\ker\left(
H^d_\m\left(\Delta^\ell/\Delta^{\ell+1}\right)
\lra
H^d_\m\left((R\otimes_k R)/\Delta^{\ell+1}\right) 
\right), 
$$
so that we have a short exact sequence
$$
0\lra K \lra 
H^d_\m\left((R\otimes_k R)/\Delta^{\ell+1}\right) \lra 
H^d_\m\left((R\otimes_k R)/\Delta^\ell\right)\lra0
$$
and a surjection
$$
H^d_\m\left(\Delta^\ell/\Delta^{\ell+1}\right)
\lra
K_\ell\lra 0.
$$
Applying graded Matlis duality to the short exact sequence, we get
$$
0\lra
H^d_\m\left((R\otimes_k R)/\Delta^\ell\right)^\star\lra
H^d_\m\left((R\otimes_k R)/\Delta^{\ell+1}\right)^\star \lra 
 K_\ell^\star\lra 0.
$$
Thus, by induction, to show the vanishing of 
$H^d_\m \left(\frac{R\otimes_k R}{\Delta^\ell}\right)^*(-a)$ in degree less than $0$ for all $\ell$, it suffices to show the vanishing of $K_\ell^\star(-a)$ in degree less than $0$ for all $\ell$.
But Matlis duality applied to the surjection above gives an injection
$$
K_\ell^\star \hookrightarrow 
H^d_\m\left(\Delta^\ell/\Delta^{\ell+1}\right)^\star,
$$
so it suffices to show the vanishing of 
$H^d_\m\left(\Delta^\ell/\Delta^{\ell+1}\right)^\star(-a)$ in degree less than $0$.

Every $\Delta^\ell/\Delta^{\ell+1}$ is naturally an $(R\otimes_k R)/\Delta$-module, but this just means each is an $R$-module. The maximal ideal $\m$ of $R\otimes_k R$  modulo $\Delta$ is just the homogeneous maximal ideal $\m_R$ of $R$, so it suffices to show the vanishing of 
$H^d_{\m_R}(\Delta^\ell/\Delta^{\ell+1})^\star(-a)$ in degree less than $0$.

Now, consider the natural surjection of $R$-modules
$$
\Sym^\ell(\Omega_R)=
\Sym^\ell(\Delta/\Delta^2)
\lra \Delta^\ell/\Delta^{\ell+1}\lra 0.
$$
Applying $H^d_{\m_R}(-)$, we get a surjection
$$
H^d_{\m_R}\left(\Sym^\ell\left(\Omega_R\right)\right)\lra H^d_{\m_R}\left(\Delta^\ell/\Delta^{\ell+1}\right)\lra 0
$$
(where surjectivity follows since we are considering the top-degree local cohomology).
Applying Matlis duality yet again, we obtain an injection
$$
H^d_{\m_R}\left(\Delta^\ell/\Delta^{\ell+1}\right)^\star
\lra 
H^d_{\m_R}\left(\Sym^\ell\left(\Omega_R\right)\right)^\star,
$$
so our desired vanishing would come from showing 
$H^d_{\m_R}(\Sym^\ell(\Omega_R))^\star(-a)$ is zero in degree less than $0$.
Finally, by using the correspondence between local cohomology with respect to the homogeneous maximal ideal and sheaf cohomology on $X=\Proj R$, as well as Serre duality, one sees immediately that this is equivalent to the vanishing
of $H^0(\Sym^\ell(\Omega_R)^\vee\otimes \L^{-e})$ for every $e>0$, $\ell>0$.

This is equivalent to 
$H^0(\Sym^\ell(\Omega_R)^\vee\otimes \L^{-1})=0$ for all $\ell$ since if for some $e,\ell$, we had the inequality  
$H^0(\Sym^\ell(\Omega_R)^\vee\otimes \L^{-e})\neq 0$, then multiplication of global sections 
$$
H^0\left(\Sym^\ell\left(\Omega_R\right)^\vee\otimes \L^{-e}\right)\otimes H^0\left(\L^{e-1}\right)\lra 
H^0\left(\Sym^\ell\left(\Omega_R\right)^\vee\otimes \L^{-1}\right)
$$
yields a nonvanishing global section of 
$ H^0(\Sym^\ell(\Omega_R)^\vee\otimes \L^{-1})$.

Finally, we claim that if 
$$
H^0\left(\left(\Sym^{m'} \wtilde \Omega_X\right)^\vee \otimes\L\inv\right)\neq0
$$
for some $m'\gg0$, then
$$
H^0\left(\left(\Sym^m \Omega_X\right)^\vee \otimes\L\inv\right)\neq0
$$
for some $m\gg0$.
Assume 
that $H^0((\Sym^m \Omega_X)^\vee \otimes\L\inv)=0$ for all $m$.
But then Lemma~\ref{extens} implies immediately that 
$H^0((\Sym^{m'} \wtilde \Omega_X)^\vee \otimes\L\inv)=0$ for all $m'$.
Thus, we must have the claimed nonvanishing, and the theorem follows.
\end{proof}

\begin{rem}
In characteristic 0, we have an isomorphism
$$
(\Sym^m  \Omega_X)^\vee\cong \Sym^m ( \Omega_X^\vee)= \Sym^m(T_X).
$$
Thus, the conclusion of our theorem is exactly the same as Hsiao's (\textit{i.e.}, that $T_X$ is big). However, in characteristic $p$, $(\Sym^m \Omega_X)^\vee $ is not the symmetric power $\Sym^m T_X$
but rather $\Gamma^m T_X$, the $\supth{m}$ \emph{divided} power of $T_X$. In many examples we have calculated, when $p\mid m$,  $H^0(\Sym^m(T_X))$ may be zero while $H^0((\Sym^m \Omega_X)^\vee)$ is nonzero. 
Heuristically, it is the nonvanishing of the latter that leads to ``more'' differential operators in characteristic $p$ than in characteristic 0: for example, cones over Fano varieties in characteristic $p$ will always have differential operators of negative degree (because they have $F$-regular coordinate rings), while this is likely rarer in characteristic $0$ (see \cite{DM1,DM2} for examples of this phenomenon).

To our knowledge, the behavior of the positivity of divided powers $\Gamma^m(E):=(\Sym^m E^\vee)^\vee$ of a vector bundle $E$ has not been well studied, in contrast to the case of symmetric powers $\Sym^m E$.
For example, the tautological line bundle $\O_{\P(E)}(1)$ ``encodes'' all symmetric powers $\Sym^m E$, in the sense that $\pi_* \O_{\P(E)}(m) =\Sym^m E$; we do not know of an analogue for divided powers.
As we will see in the following, for a given $X$, a better understanding of the divided powers of $T_X$ (and in particular, their ``nonpositivity'') would lead to the failure of FFRT for the homogeneous coordinate ring of~$X$.
\end{rem}

\begin{rem}
If $R$ is a strongly $F$-regular domain, then \cite[Theorem~2.2]{Smith} implies that $R$ is $D$-simple, \textit{i.e.}, that $R$ is a simple module under the action of the ring of differential operators $D_R$.
If $R$ is moreover graded, this implies immediately that $D_R$ has elements of negative degree.
To our knowledge, the ring of differential operators is less well understood outside the $F$-regular/$F$-pure setting. In what follows, as a corollary of our proof of non-FFRT for the rings we consider, we obtain that these rings are not $D$-simple either. Many of these rings are not $F$-pure and thus previously inaccessible to existing methods of study. Thus, our results have implications for the study of differential operators in characteristic $p$, beyond just the FFRT property.
\end{rem}

\section{FFRT for cones}
With the work of the preceding sections done, the following theorem is now immediate. 

\begin{thm}
\label{main}
Let $X$ be a variety over a perfect field $k$ of characteristic $p>0$ and $\L$ an ample line bundle on $X$ such that $\L^{\otimes r}\cong \w_X$ for some $r\in \Z$ and  $H^i(X,\L^{m})=0$ for $i=1,\dots,\dim X-1$ and $m\in \Z$.
If $H^0((\Sym^m X)^\vee\otimes \L\inv)=0$ for all $m$, then the homogeneous coordinate ring
$\bigoplus H^0(X,\L^{m})$ does not have \graded FFRT.
\end{thm}

\begin{rem}
Note that the condition that
 $H^i(X,\L^{m})=0$ for $i=1,\dots,\dim X-1$ and $m\in \Z$ ensures that the section ring $\bigoplus H^0(X,\L^m)$ is Cohen--Macaulay.
It holds automatically if, for example, $X$ is a complete intersection in projective space.

In addition, this assumption can be weakened to asking only that $H^i(X,\O_X)=0$ for (locally) Cohen--Macaulay (\textit{e.g.}, smooth) varieties with $\w_X=\O_X$:
To see this,
note that we may replace $\L$ with an arbitrarily positive multiple $\L^{\otimes d}$ and preserve the condition that $\L^{\otimes r}\cong \w_X\cong \O_X$ for some $r\in\Z$.
The Veronese subring 
$$
\bigoplus H^0\left(X,\L^{\otimes md}\right) 
$$
is a graded direct summand of 
$$
\bigoplus H^0\left(X,\L^{\otimes m}\right),
$$
and if the latter has FFRT, then the former does as well by 
\cite[Proposition~3.1.6]{SVdB}.
By Serre vanishing, there is then some positive $d$ such that
for all positive $m$ and all $i>0$, we have the vanishing $H^i(X,\L^{dm})=0$.
For $m$ negative, $H^i(X,\L^{md})$ is Serre dual to $H^{\dim X-i}(X,\L^{-md})$, and since we only consider the range $1<i<\dim X$, this  also vanishes for any $m$ and some fixed $d$ large enough. 
Thus, after replacing $\L$ by $\L^d$, we see that the vanishings $H^i(X,\L^m)$ automatically hold for $m\neq 0$, and thus only the vanishings $H^i(X,\O_X)$ are necessary.
\end{rem}

\begin{proof}[Proof of Theorem~\ref{main}]
Note that $X$ is arithmetically quasi-Gorenstein, by the condition that $\L^{\otimes r} \sim \w_X$ for some $r\in \Z$; our assumption on $H^i(X,\L^m)$ guarantees that $X$ is arithmetically Cohen--Macaulay as well, and thus $X$ is arithmetically Gorenstein.
By Theorem~\ref{diffops}, the vanishing 
$H^0((\Sym^m X)^\vee\otimes \L\inv)=0$ ensures that $R$ has no differential operators of negative degree.
Then Corollary~\ref{fgnegcor} ensures that $R$ does not have \graded FFRT.
\end{proof}

\begin{rem}
If $\L^{\otimes r}\not\cong \w_X$ for any $r\in \Z$,
or if $X$ is not arithmetically Cohen--Macaulay,
it is not clear how differential operators on $R(X,\L')$ are related to the (non)vanishing of 
$H^0((\Sym^m X)^\vee\otimes \L\inv)=0$ since Theorem~\ref{diffops} required that $R$ be Gorenstein. It would be very useful to have a similar statement for general polarizations $\L'$. Put another way, it would be useful to understand how both the differential operators on $R(X,\L)$ and the FFRT property for $R(X,\L)$ vary with the choice of line bundle $\L$.
\end{rem}

\begin{rem}
A natural question then is when the vanishing 
$H^0((\Sym^m X)^\vee\otimes \L\inv)=0$ holds. 
In the next few sections, we show this holds for $X$ a Calabi--Yau variety or complete intersection of general type (this is harder).
Similarly, it holds automatically for abelian varieties (although when $\dim X>1$, these will never satisfy the vanishing $H^i(X,\O_X)=0$ for $1<i<\dim X$, so fall outside the scope of our theorem).
It is likely that this vanishing holds in broader generality: for example, one can ask if it holds for \emph{all} varieties of general type. If $\Omega_A$ is ``positive'' (in some sense), then one might expect $(\Sym^m \Omega_A)^\vee$ to be negative, and thus to have no global sections (without even having to twist by $\L\inv$).
If one can make this precise, one can immediately obtain the failure of \graded FFRT for homogeneous coordinate rings of a much broader class of variety.
Note however that this requires strong positivity conditions on $\Omega_A$: it is not enough, for example, to be effective or even big since a nontrivial vector bundle and its dual may both have nonzero sections or be big, unlike the case for line bundles.
In Section~\ref{generaltype}, we make use of the strong semistability properties for complete intersections of general type, but we cannot expect that this holds for all general-type varieties (for example, even regular semistability will fail for the product of two general-type varieties with cotangent bundles of different slopes).
\end{rem}

Theorem~\ref{main} immediately recovers the following theorem. 

\begin{thm}
Let $X$ be a smooth curve of genus $g\geq 1$. Then the homogeneous coordinate ring of $X$ under any embedding does not have \graded FFRT.
\end{thm}

\begin{proof}
In this case, the vanishing
$(\Sym^m\Omega_X)^\vee\otimes \L\inv=0$ is easily seen to hold: if  $X$ is a smooth curve of genus $g\geq 1$, then $(\Sym^m\Omega_X)^\vee$ is a line bundle of degree $m(2-2g)$ and thus has no global sections, even before twisting by the negative-degree $\L\inv$.
\end{proof}

\begin{rem}
This theorem has been known for some time:
it follows from \cite{Tango} in the $g=1$ case and \cite{LP} for $g\geq 2$.
In either case, it follows from a study of the behavior of vector bundles on such curves under the Frobenius morphism. Here, we have obtained it instead as a consequence of the positivity of \emph{line} bundles, with no reference to the Frobenius morphism, thus providing a more streamlined (though perhaps less illustrative) proof.
\end{rem}

\section{K3 surfaces and Calabi--Yau varieties}
\label{K3}

In this section, we treat \graded FFRT of homogeneous coordinate rings of varieties $X$ with $\w_X$ trivial; we show that unless such a variety is ``close to rational,''  its homogeneous coordinate ring cannot have \graded FFRT.
The crucial ingredient in the results of this section is \cite{Langer}, which analyzes the stability of (co)tangent bundles of Calabi--Yau varieties and K3 surfaces in positive characteristic.

Fix a perfect field $k$ of characteristic $p>0$.
We begin by recalling a few relevant definitions. 

\begin{dfn}
Let $X$ be a variety of dimension $n$ over $k$. We say that $X$ is unirational if it admits a dominant rational map $\P^n \dra X$, and uniruled if it admits a dominant rational map $Y\times \P^1 \dra X$ for some $Y$ of dimension $n-1$.
\end{dfn}

Note that unirationality implies uniruledness.

\begin{dfn}
Let $X$ be a smooth variety of dimension $n-1$, and let $H$ be an ample divisor. For a coherent sheaf $E$ on $X$, we set 
$$
\mu_H(E) := \frac{c_1(E)\cdot H^{n-1}}{\rank E}.
$$
When $H$ is fixed, we often write just $\mu$.

We say that a torsion-free sheaf $E$ is $\mu$-semistable if for all subsheaves $F\subset E$, we have $\mu(F)\leq \mu(E)$.
Equivalently, $E$ is $\mu$-semistable if for all torsion-free quotients $E\to F$, we have $\mu(E)\leq \mu(F)$.

We say that $E$ is strongly $\mu$-semistable if for all $e$, the Frobenius pullback $F^*_e E$ is semistable.
\end{dfn}

We often say just ``semistable'' and ``strongly semistable'' when $\mu$ is fixed.

We will use the following easy observation. 

\begin{lem}
\label{dual}
The dual of a strongly semistable locally free sheaf is strongly semistable.
\end{lem}

\begin{proof}
Taking duals commutes with pullbacks for locally free sheaves of finite rank, so that $(F_e^*(E))^\vee \cong F^*_e(E^\vee)$. Since duals of semistable sheaves are semistable and $F$ is finite by our assumptions on perfectness of the ground field, the lemma follows.
\end{proof}

The following lemma is nontrivial, and is crucial in what follows.

\begin{lem}
\label{powers}
Symmetric powers of strongly semistable sheaves are strongly semistable.
\end{lem}

Without the ``strong'' hypothesis, this is true in characteristic 0
but false in positive characteristic.  This is due originally to
\cite[Theorem~3.23]{RR}.  For a nice exposition, see also
\cite[Section~4.2]{Langer2}.

Finally, we will need the following lemma. 

\begin{lem}
\label{sections}
Let $E$ be a $\mu$-semistable vector bundle with $\mu(E)\geq 0$ and $\L$ an ample line bundle. Then $H^0(E^\vee\otimes \L\inv)=0 $.
\end{lem}

\begin{proof}
Assume the conclusion does not hold; we then have a global section $\O_X\to E^\vee\otimes \L\inv$. Twisting by $\L$ and dualizing, we get a nonzero map $E\to \L\inv$. Let $\F$ be the image of $E\to \L\inv$. Then $\F\subset \L\inv$, so we must have that $\mu(\F)\leq \mu(\L\inv)<0$. But since there is a surjection $E\to \F$, the $\mu$-semistability of $E$ ensures that $\mu(E)\leq \mu(F)$. We thus have that $0\leq \mu(E)\leq \mu(F)\leq \mu(L\inv)<0$, giving a contradiction.
\end{proof}

\begin{thm}
\label{K3proof}
Let $X$ be a Calabi--Yau variety of dimension $n$ with $H^i(X,\O_X)=0$ for $1<\dim X<n$ over a perfect field $k$ of characteristic $p\geq (n-1)(n-2)$.
If\, $X$ is not uniruled,
then for any ample line bundle $\L$, the homogeneous coordinate ring $R(X, \L)$  does not have \graded FFRT.
\end{thm}

Again, we note that the condition on $H^i(X,\O_X)$ is automatic for Calabi--Yau varieties arising as complete intersections in projective space. 

\begin{proof}
Fix an ample polarization $H$ of $X$.
By \cite[Theorem~0.1]{Langer},  
since $\w_X$  is numerically trivial and $X$ is not uniruled, and $p\geq (n-1)(n-2)$,
the tangent bundle of $X$ is strongly $\mu$-semistable.  
Thus, by Lemma~\ref{dual}, so is $\Omega_X$. Moreover, it is clear that $\mu(\Omega_X)=0$ since $c_1(\Omega_X)=c_1(\w_X)=c_1(\O_X)=0$. By Lemma~\ref{powers}, $\Sym^m \Omega_X$ is (strongly) semistable, as is $(\Sym^m \Omega_X)^\vee$, and clearly $\mu((\Sym^m \Omega_X)^\vee )=0$ as well.
Then Lemma~\ref{sections} implies that 
$$
H^0\left(\left(\Sym^m \Omega_X\right)^\vee \otimes \L\inv\right)=0,
$$
so that $R(X,\L)$ cannot have \graded FFRT.
\end{proof}

For K3 surfaces, we can weaken the condition ``not uniruled'' to ``not unirational,'' and we automatically have the vanishing of $H^1(X,\O_X)$ by definition. 

\begin{thm}
Let $X$ be a K3 surface over a perfect field $k$ of characteristic $p>0$ $($\textit{i.e.}, $X$ is a smooth surface with $\omega_X\cong \O_X\,)$.
If\, $X$ is not unirational,
then for any ample line bundle $\L$, the homogeneous coordinate ring $R(X, \L)$  does not have \graded FFRT.
\end{thm}

\begin{proof}
The proof proceeds exactly as the previous one, except that \cite[Proposition~4.1]{Langer} ensures that if the tangent bundle of a K3 surface $X$ is not strongly semistable with respect to an ample polarization $H$ of $X$, then $X$ is a Zariski surface,\footnote{A Zariski surface is a surface admitting a purely inseparable dominant rational map of degree $p$ from $\P^2$; in particular, such a surface is unirational.} and thus unirational, not just uniruled. (We also do not need to use Lemma~\ref{dual} since for a K3 surface, $\Omega_X\cong T_X$.)
\end{proof}

\begin{exa}
\label{fermat}
By \cite[Section~1, Corollary]{Shioda}, the Fermat quartic $V(x^4+y^4+z^4+w^4)$ is unirational in characteristic $p\equiv 3\mod4 $ but not in $p\equiv 1 \mod 4$. Therefore, if $k$ is a perfect field of characteristic $p\equiv 1 \mod 4$, the theorem implies the ring
$$
\frac{k[x,y,z,w]}{x^4+y^4+z^4+w^4}
$$
does not have \graded FFRT.
\end{exa}

\section{Complete intersections of general type}
\label{generaltype}

We now turn to the case of complete intersections of general type, where we have the following. 

\begin{thm}
\label{gentype}
Let $X\subset \P^{n+c}$ be an $n$-dimensional smooth complete intersection of multidegree $(d_1,\dots,d_c)$ over a perfect field $k$. If\, $\sum d_i > n+c$ $($\textit{i.e.}, if\, $X$ is not Fano$)$ and $\Pic(X)=\Z\cdot \O_X(1)$, then the homogeneous coordinate ring of\, $X$ $($in any embedding\,$)$ does not have \graded FFRT.
\end{thm}

\begin{proof} Let $R$ be the homogeneous coordinate ring of the complete intersection $X$. If $R$ had FFRT, the same holds for $R\otimes_k \bar k$, so we may harmlessly assume that $k=\bar k$. We may then apply the results of \cite[Proposition~4.2]{Noma1}
and \cite[Theorem~1.1]{Noma2}, which together say that for smooth complete intersections $X$ over an algebraically closed field $k$ of characteristics $p>0$ such that $\Pic(X)=\Z\cdot \O_X(1)$, the tangent and cotangent bundles of $X$ are strongly semistable. We then apply the exact same argument as in the proof of Theorem~\ref{K3proof}: the symmetric powers $\Sym^m \Omega_X$ and their duals $(\Sym^m \Omega_X)^\vee$ are semistable as well. Moreover, we have that $\mu(\Omega_X)\geq 0$ (with equality exactly when $\sum d_i = n+c+1$). Thus, we have immediately by Lemma~\ref{sections} that $H^0((\Sym^m \Omega_X)^\vee)\otimes \L\inv)=0$ for any ample line bundle $\L$.
\end{proof}

\begin{exa}
Note that this also applies to Calabi--Yau complete intersections. In the previous section, we showed that non-uniruled Calabi--Yau varieties will not have \graded FFRT. As most Calabi--Yau varieties have no reason to be complete intersections, the results in the preceding section are more general. However, the results in this section do  not need non-uniruledness, which
is quite helpful in giving explicit examples:
For example, Theorem~\ref{gentype}
immediately implies that  the Fermat quintic threefold
$$
\frac{k[x,y,z,w,t]}{x^5+y^5+z^5+w^5+t^5}
$$
does not have \graded FFRT for any perfect field $k$ of positive characteristic.
(Recall that for the corresponding Fermat quartic surface considered in Example~\ref{fermat}, we only had the conclusion in characteristics $p\equiv1\mod3$.)
The same reasoning applies for any general hypersurface of degree at least $n+1$ in $\P^n$, for $n\geq 4$, without having to consider the subtler question of uniruledness at all.
\end{exa}

\begin{rem}
The condition that $\Pic(X)=\Z\cdot \O_X(1)$ is automatic when $\dim X\geq 3$, by the Grothendieck--Lefschetz theorem. 
Moreover,
when $\dim X=2$ and $\sum d_i>n+c$, \cite[Expos\'e~XIX, Th\'eor\`eme 1.2]{SGA7II} showed that over an algebraically closed field, $\Pic(X)=\Z\cdot \O_X(1)$ for very general complete intersections $X$ of type $(d_1,\dots,d_c)$ (in fact, this is true whenever $(d_i)\neq (2), (3), (2,2)$, though we do not need these cases). However, over a finite field $k$, this is much more subtle, and in fact it is possible for $\Pic(X)\neq \Z\cdot \O_X(1)$ for every degree $d$ hypersurface $X$ defined over a finite field $k$. For a more thorough discussion of these results, see \cite[Section~1]{Ji}.

\end{rem}

\begin{rem}
In fact, Theorem~\ref{gentype} is true for smooth weighted complete intersections in weighted projective space as well, by the exact same proof we give, since the results of \cite{Noma1,Noma2} apply for smooth weighted complete intersections of Picard number 1.
\end{rem}

\section{Questions}
\label{questions}

In general, it is not clear how precisely the FFRT property is related to classes of $F$-singularities. As mentioned above, strongly $F$-regular rings need not have FFRT (by \cite{SinghSwanson,TT}). Conversely, rings with FFRT need not be strongly $F$-regular, or even $F$-pure. In \cite{HO}, Hara and Ohkawa give examples of rings with FFRT that are not even $F$-pure: for example, if $k$ is a field of characteristic $p \in \set{2,3,7}$, then
$$
k[x,y,z]\,/(x^2+y^3+z^7)
$$
has FFRT but is not F-pure.
However, if $\Char k \notin \set{2,3,7}$, then this ring does not have FFRT (and is not $F$-pure).
Put another way, if $R=\Z[x,y,z]\,/(x^2+y^3+z^7)$, then for almost all reductions $R_p:=R/p$, $R_p$ will fail to have FFRT (and fail to be $F$-pure as well).
 One is thus led to the following question.

\begin{quest}
Let $R$ be a normal ring that is not of strongly $F$-pure type. Can ``most'' reductions $R_p$ have FFRT? (Where ``most'' can mean either a dense set of primes $p$ or an open dense set.) If ``normal'' is omitted, Jack Jeffries has pointed out that the answer is positive: $k[t^2,t^3]$ will be FFRT for any field $k$ with $\Char k > 3$, but $\C[t^2,t^3]$ is not of strongly $F$-pure type.
\end{quest}

\begin{rem}
As pointed out by Anurag Singh, Stanley--Reisner rings are of $F$-pure type and have FFRT, so if we weaken ``not $F$-pure type'' to ``not strongly $F$-regular type,'' the above has a positive answer. We are not aware of a similar example where $R$ is a domain, however.
\end{rem}

Our results above give many examples of rings that do not have differential operators of negative degree and thus cannot be $D$-simple. These rings have $F$-pure or worse singularities; in contrast, strongly $F$-regular rings will be $D$-simple by \cite{Smith}.
However, it is possible for non-$F$-pure rings to be $D$-simple, as the following examples show. 

\begin{exa}
For any field of characteristic $>3$,
$R=k[t^2,t^3]$ is $D$-simple but not $F$-pure. Note that this ring is not normal, but it can be defined uniformly across all characteristic $p>3$.
\end{exa}

\begin{rem}
The following example, related to us by Jack Jeffries, furnishes an example of a \emph{normal} ring that is not $F$-pure but  is $D$-simple.
Let $\Char k =p>0$, and assume that $k$ contains a root of $T^p-T-1$.
Consider the hypersurface ring
$$
R=
\frac{
k[x_1,x_2,M_1,M_2,h]}{\left(
h^p-(x_1^{p-1}-x_2^{p-1})M_1+x_1^{p^2-p}M_2-(x_1^p(x_1^{p-1}-x_2^{p-1}))^{p-1}h
\right)}
$$ 
of \cite[Example~3.1]{Dufresne} (whose notation we follow).
The ring $R$ is not $F$-pure since $h$ is in the Frobenius closure of $(x_1,x_2,M_1,M_2)$ but not in the ideal itself (as can be seen from the defining equation).
It can be immediately verified that this is regular in codimension 1 and thus normal.

However, this ring is $D$-simple: The discussion in \cite[Example~3.1]{Dufresne} shows that there is a radicial morphism from $k[t_1,\dots,t_4]$ to $R$; in the terminology of \cite[Section~5, p.~4932]{Jeffries}, $R$ has radicial rank $4$. Since this is the same as the dimension of $R$, \cite[Propositions~5.2 and~3.2]{Jeffries} imply that $R$ is $D$-simple.
\end{rem}

Note that this example concerns a particular prime $p$, and there is not a single ring defined in characteristic~0 whose reductions give this example for multiple primes $p$.
The following question is then natural, and has been considered by many others throughout the years, though we do not know if it appears explicitly in the literature. 

\begin{quest}
Are there examples of normal rings defined over $\C$, for which ``most'' reductions modulo $p$ (either an open dense set or just a dense set) are $D$-simple but not $F$-pure?
\end{quest}

\begin{quest}
In Section~\ref{K3}, our results needed the assumption that the Calabi--Yau variety in question was not uniruled. Can uniruled Calabi--Yau varieties have homogeneous coordinate rings with FFRT? 

Note that both rationality properties of $X$ and the behavior of $F_*^e\O_X$ reflect sensitive arithmetic properties:
For example,
a conjecture of Artin \cite{Artin} (in conjunction with the now-proven Tate conjecture) implies that supersingular K3 surfaces are unirational, and a conjecture of Shioda \cite{Shiodacon} says that simply connected surfaces are unirational if and only if they are supersingular, while \cite{ST} shows that the decomposition of $F_*^e\O_A$ for an abelian variety $A$ depends on the $p$-rank of $A$, \textit{e.g.}, on whether $A$ is supersingular. So, it would not be completely surprising if FFRT depended on delicate questions of unirationality. As a first example, it would be interesting to know the following:
Does
$$
\frac{\mathbb F_3 [x,y,z,w]}{x^4+y^4+z^4+w^4}
$$
have FFRT? Note that this is the simplest example of \cite{Shioda} of a unirational K3 surface and that if $\mathbb F_3$ were replaced by $\mathbb F_p$ for $p\equiv 1 \mod 4$, then the answer would be negative.
\end{quest}

Finally, we note that our results all proceed by proof by contradiction: the lack of global sections of $(\Sym^m \Omega_X)^\vee \otimes \L\inv$ means that $R[1/x]$ is not generated over $D_R$ by $1/x$, and thus the section ring $R(X,\L)$ cannot have \graded FFRT.
In many cases, however, we are interested not just in how many classes of indecomposable summands occur in $F_*^e R$ as $e$ varies, but also in what these indecomposable summands actually are.
We can thus ask for a more detailed understanding of the failure of FFRT. 

\begin{quest}
Is there a way of explicitly constructing infinitely many classes of indecomposable summands of the $F_*^eR$ from the failure of $R[1/x]$ to be generated over $D_R$ by $1/x$, for particular choices of $x$? 
\end{quest}


\providecommand{\bysame}{\leavevmode\hbox to3em{\hrulefill}\thinspace}


\begin{thebibliography}{R{\v{S}}VdB19+++}

\bibitem[Ach15]{Achinger}
P.~Achinger, \emph{A characterization of toric varieties in characteristic
  {$p$}}, Int.\ Math.\ Res.\ Not.\ IMRN (2015), no.~16, 6879--6892.

\bibitem[Art74]{Artin}
M.~Artin, \emph{Supersingular {$K3$} surfaces}, Ann.\ Sci.\ \'{E}cole Norm.\ Sup.~(4) \textbf{7} (1974), 543--567.

\bibitem[BDO08]{BDO}
F.~Bogomolov and B.~De~Oliveira, \emph{Symmetric tensors and geometry of {$\Bbb
  P^N$} subvarieties}, Geom.\ Funct.\ Anal.\ \textbf{18} (2008), no.~3, 637--656.

\bibitem[BEH87]{BE}
R.-O.~Buchweitz, D.~Eisenbud, and J.~Herzog, \emph{Cohen-{M}acaulay modules on
  quadrics}, in: \emph{Singularities, representation of algebras, and vector bundles} ({L}ambrecht, 1985),  pp.~58--116, Lecture Notes in Math., vol.~1273, Springer, Berlin,  1987.

\bibitem[DQ20]{DQ}
H.~Dao and P.\,H.~Quy, \emph{On the associated primes of local cohomology},
  Nagoya Math.~J.\ \textbf{237} (2020), 1--9.

\bibitem[DK73]{SGA7II}
P.~Deligne and N.~Katz (eds), \emph{S\'eminaire de G\'eom\'etrie alg\'ebrique du
Bois-Marie 1967--1969 (SGA 7 II), Groupes de monodromie en g\'eom\'etrie
alg\'ebrique. II}, Lecture Notes in Math., vol.~340, Springer-Verlag, Berlin-New York , 1973.

\bibitem[Duf09]{Dufresne}
E.~Dufresne, \emph{Separating invariants and finite reflection groups}, Adv.\
  Math.\ \textbf{221} (2009), no.~6, 1979--1989.

\bibitem[Har15]{Hara1}
N.~Hara, \emph{Looking out for frobenius summands on a blown-up surface of
{P2}}, Illinois J.~Math.\ \textbf{59} (2015), 115--142.

\bibitem[HO20]{HO}
N.~Hara and R.~Ohkawa, \emph{The {FFRT} property of two-dimensional normal
  graded rings and orbifold curves}, Adv.\ Math.\ \textbf{370} (2020), 107215.

\bibitem[Har77]{Hartshorne}
R.~Hartshorne, \emph{Algebraic {Geometry}}, Grad.\ Texts in Math.,
  vol.~52, Springer, New York, NY, 1977.

\bibitem[HNB17]{HNB}
M.~Hochster and L.~{N\'{u}\~{n}ez}-Betancourt, \emph{Support of local
  cohomology modules over hypersurfaces and rings with {FFRT}}, Math.\ Res.\ 
  Lett.\ \textbf{24} (2017), no.~2, 401--420.

\bibitem[Hsi15]{Hsiao}
J.-C. Hsiao, \emph{A remark on bigness of the tangent bundle of a smooth
  projective variety and {$D$}-simplicity of its section rings}, J.\ Algebra
  Appl.\ \textbf{14} (2015), no.~7, 1550098.

\bibitem[Jef21]{Jeffries}
J.~Jeffries, \emph{Derived functors of differential operators}, Int.\ Math.\ Res.\ Not.\ IMRN (2021), no.~7, 4920--4940.

\bibitem[Ji21]{Ji}
L.~Ji, \emph{The {N}oether--{L}efschetz theorem}, preprint \arXiv{2107.12962} (2021). To appear in J. Algebraic Geom.

\bibitem[LP08]{LP}
H.~Lange and C.~Pauly, \emph{On {F}robenius-destabilized rank-2 vector bundles
over curves}, Comment.\ Math.\ Helv.\ \textbf{83} (2008), no.~1, 179--209.

\bibitem[Lan09]{Langer2}
A.~Langer, \emph{Moduli spaces of sheaves and principal {$G$}-bundles},
 in: \emph{Algebraic geometry---{S}eattle 2005. {P}art 1}, pp.~273--308, Proc.\ Sympos.\ Pure Math., vol.~80, Amer.\ Math.\ Soc., Providence, RI, 2009.

\bibitem[Lan15]{Langer}
\bysame, \emph{Generic positivity and foliations in positive characteristic},
  Adv.\ Math.\ \textbf{277} (2015), 1--23.

\bibitem[Mal21]{DM1}
D.~Mallory, \emph{Bigness of the tangent bundle of del {P}ezzo surfaces and
  {$D$}-simplicity}, Algebra Number Theory \textbf{15} (2021), no.~8,
  2019--2036.

\bibitem[Mal22]{DM2}
\bysame, \emph{Homogeneous coordinate rings as direct summands of regular
  rings}, preprint \arXiv{2206.03621} (2022). To appear in Ill. J. Math.

\bibitem[Nom97]{Noma1}
A.~Noma, \emph{Stability of {F}robenius pull-backs of tangent bundles and
  generic injectivity of {G}auss maps in positive characteristic}, Compos.\
  Math.\ \textbf{106} (1997), no.~1, 61--70.

\bibitem[Nom01]{Noma2}
\bysame, \emph{Stability of {F}robenius pull-backs of tangent bundles of
  weighted complete intersections}, Math.\ Nachr.\ \textbf{221} (2001), 87--93.

\bibitem[R{\v{S}}VdB22]{RSB}
T.~Raedschelders, {\v{S}}.~{\v{S}}penko, and M.~Van~den Bergh, \emph{The
  {F}robenius morphism in invariant theory {II}}, Adv. Math.~\textbf{410} Part~A (2022), article ID~108587. 

\bibitem[RR84]{RR}
S.~Ramanan and A.~Ramanathan, \emph{Some remarks on the instability flag},
  Tohoku Math.\ J.\ (2) \textbf{36} (1984), no.~2, 269--291.

\bibitem[ST16]{ST}
A.~Sannai and H.~Tanaka, \emph{A characterization of ordinary abelian varieties
  by the {F}robenius push-forward of the structure sheaf}, Math.\ Ann.\ 
\textbf{366} (2016), no.~3-4, 1067--1087.

\bibitem[Shi11]{Shibuta}
T.~Shibuta, \emph{One-dimensional rings of finite {$F$}-representation type},
  J.~Algebra \textbf{332} (2011), 434--441.

\bibitem[Shi74]{Shioda}
T.~Shioda, \emph{An example of unirational surfaces in characteristic {$p$}},
  Math.\ Ann.\ \textbf{211} (1974), 233--236.

\bibitem[Shi77]{Shiodacon}
\bysame, \emph{Some results on unirationality of algebraic surfaces}, Math.\
Ann.\ \textbf{230} (1977), no.~2, 153--168.

\bibitem[SS04]{SinghSwanson}
A.\,K.~Singh and I.~Swanson, \emph{Associated primes of local cohomology modules
and of {F}robenius powers}, Int.\ Math.\ Res.\ Not.\ (2004), no.~33, 1703--1733.

\bibitem[Smi95]{Smith}
K.\,E.~Smith, \emph{The {D}-module {Structure} of {F}-{Split} {Rings}},
  Math.\ Res.\ Lett.~\textbf{2} (1995), no.~4, 377--386.

\bibitem[SVdB97]{SVdB}
K.\,E.~Smith and M.~Van~den Bergh, \emph{Simplicity of rings of differential
  operators in prime characteristic}, Proc.\ London Math.\ Soc.~(3) \textbf{75}
  (1997), no.~1, 32--62.

\bibitem[TT08]{TT}
S.~Takagi and R.~Takahashi, \emph{{$D$}-modules over rings with finite
  {$F$}-representation type}, Math.\ Res.\ Lett.\ \textbf{15} (2008), no.~3,
563--581.

\bibitem[Tan72]{Tango}
H.~Tango, \emph{On the behavior of extensions of vector bundles under the
  {F}robenius map}, Nagoya Math.~J.\ \textbf{48} (1972), 73--89.

\bibitem[Tho00]{Thomsen}
J.\,F.~Thomsen, \emph{Frobenius direct images of line bundles on toric
  varieties}, J.~Algebra \textbf{226} (2000), no.~2, 865--874.

\bibitem[Yek92]{Yekutieli}
A.~Yekutieli, \emph{An explicit construction of the {G}rothendieck residue
complex} (with an appendix by P.~Sastry), Ast\'{e}risque \textbf{208} (1992).

\end{thebibliography}
\end{document}